\documentclass[11pt]{article}
\usepackage{amssymb,amsfonts}
\usepackage{amsmath,amsthm,amsxtra}

\usepackage{graphicx}
\usepackage{amscd}
\usepackage{ulem}
\usepackage{makeidx}
\usepackage{fancybox}

\usepackage[linktocpage=true, colorlinks=true, linkcolor=blue, citecolor=red, urlcolor=green]{hyperref}

\newcommand{\ccc}{{\mathbf C}}

\newcommand{\nnn}{{\mathbf N}}

\newcommand{\zzz}{{\mathbf Z}}

\newtheorem{prop}{Proposition}[section]
\newtheorem{lemma}{Lemma}[section]

\numberwithin{equation}{section}

\begin{document}

\title{A note on Appell's functions related to the denominators of 
affine Lie superalgebras $\widehat{sl}(2|1)$ and 
$\widehat{osp}(3|2)$}

\author{\footnote{12-4 Karato-Rokkoudai, Kita-ku, Kobe 651-1334, 
Japan, \qquad
wakimoto.minoru.314@m.kyushu-u.ac.jp, \hspace{5mm}
wakimoto@r6.dion.ne.jp 
}{ Minoru Wakimoto}}

\date{\empty}

\maketitle

\begin{center}
Abstract
\end{center}

In this paper we consider the simplest class of Appell's functions,
and obtain their explicit formulas and modular transformation 
properties and asymptotic behaviors.


\section{Introduction}
\label{sec:Introduction}


Functions defined for $(m,s) \in \frac12\nnn \times \frac12 \zzz$ \, 
by 
\begin{subequations}
{\allowdisplaybreaks
\begin{eqnarray}
\Phi^{(\pm)[m,s]}_1(\tau,z_1,z_2,t) \hspace{2mm}
&:=& 
e^{-2\pi imt} \sum_{j \in \zzz} \, 
(\pm 1)^j \, \frac{e^{2\pi imj(z_1+z_2)+2\pi isz_1}
q^{mj^2+sj}}{1-e^{2\pi iz_1}q^j}
\label{App:eqn:2023-503a1}
\\[1mm]
\Phi^{(\pm)[m,s]}_2(\tau,z_1,z_2,t) \hspace{2mm}
&:=&
e^{-2\pi imt} \sum_{j \in \zzz} \, 
(\pm 1)^j \, \frac{e^{-2\pi imj(z_1+z_2)-2\pi isz_2}
q^{mj^2+sj}}{1-e^{-2\pi iz_2}q^j}
\label{App:eqn:2023-503a2}
\\[1mm]
\Phi^{(\pm)[m,s]}(\tau,z_1,z_2,t) \hspace{2mm}
&:=& 
\Phi^{(\pm)[m,s]}_1(\tau,z_1,z_2,t) 
\, - \, 
\Phi^{(\pm)[m,s]}_2(\tau,z_1,z_2,t) 
\label{App:eqn:2023-503a3}
\\[2mm]
\Phi^{(\pm)[m,s] \, \ast}(\tau,z_1,z_2,t) 
&:=&
\Phi^{(\pm)[m,s]}_1(\tau,z_1,z_2,t) 
\, + \, 
\Phi^{(\pm)[m,s]}_2(\tau,z_1,z_2,t) 
\label{App:eqn:2023-503a4}
\end{eqnarray}}
\end{subequations}

\vspace{-2mm}

\noindent
are mock theta functions of rank 1.
We write simply $\Phi^{[m,s]}_i$ and $\Phi^{[m,s]}$ for 
$\Phi^{(+)[m,s]}_i$ and $\Phi^{(+)[m,s]}$ respectively.

The aim of this paper is to show that $\Phi^{[1,0]}_1(\tau,z_1,z_2,0)$
and $\Phi^{(-)[\frac12,\frac12]}_1(\tau,z_1,z_2,0)$, with suitable 
specialization of $(z_1,z_2)$, can be written explicitly in terms of 
the Mumford's theta functions.
Our method is very simple as follows. Functions 
$\Phi^{[1,0]}(\tau,z_1,z_2,0)= \big[\Phi^{[1,0]}_1-\Phi^{[1,0]}_2\big]
(\tau,z_1,z_2,0)$ and 
$\Phi^{(-)[\frac12,\frac12] \, \ast}(\tau,z_1,z_2,0)= 
\big[\Phi^{(-)[\frac12,\frac12]}_1+\Phi^{(-)[\frac12,\frac12]}_2\big]
(\tau,z_1,z_2,0)$ are known by the denominator identities of the 
affine Lie superalgebras $\widehat{sl}(2|1)$ and $\widehat{osp}(3|2)$ 
respectively. Then if we know explicit formulas for 
$\big[\Phi^{[1,0]}_1+\Phi^{[1,0]}_2\big](\tau,z_1,z_2,0)$ and 
$\big[\Phi^{(-)[\frac12,\frac12]}_1-\Phi^{(-)[\frac12,\frac12]}_2\big]
(\tau,z_1,z_2,0)$ for some $(z_1,z_2)$, then we can know explicit 
formulas for $\Phi^{[1,0]}_1(\tau,z_1,z_2,0)$ and 
$\Phi^{(-)[\frac12,\frac12]}_1(\tau,z_1,z_2,0)$ for such $(z_1,z_2)$.
In this paper, we will show that this method works well for 
$(z_1,z_2)=(z+a\tau+b, z-a\tau-b)$ where $a,b \in \frac12 \zzz$.

In this paper we use notations from \cite{W2022a}, 
\cite{W2022b}, \cite{W2022c}, \cite{W2022d}, \cite{W2022e}, 
\cite{W2022f} and \cite{W2023a}.

\section{$\Phi^{[1,0]}_1(\tau, \, z+a\tau+b, \, z-a\tau-b, \, 0)$}
\label{sec:Phi1(10)}



\begin{lemma} \,\ 
\label{App:lemma:2023-501a}
Let $a \, \in \, \frac12 \zzz_{\geq 0}$ and $b \, \in \, \frac12 \zzz$. Then
\begin{enumerate}
\item[{\rm 1)}] \,\ $\Phi^{[1,0]}_1(\tau, \, z+a\tau+b, \, z-a\tau-b, \, 0)$
\begin{equation} 
= \, 
\dfrac{- \, i}{2} \cdot \dfrac{\eta(\tau)^3 \, \vartheta_{11}(\tau, 2z)}{
\vartheta_{11}(\tau, \, z+a\tau+b) \, 
\vartheta_{11}(\tau, \, z-a\tau-b)}
\, + \, 
\frac12 \, q^{a^2} \, 
\sum_{k=0}^{4a} \, (-1)^{2bk} \, 
q^{-\frac14(k-2a)^2} \, \theta_{k,1}(\tau,2z)
\label{App:eqn:2023-502a1}
\end{equation}
\item[{\rm 2)}] \,\ $\sum\limits_{j \in \zzz} \, 
\dfrac{e^{4\pi ijz}q^{j^2}}{1-(-1)^{2b} \, e^{2\pi iz} \, q^{j+a}}$
\begin{equation} 
= \, 
\frac{- \, i}{2} \cdot \dfrac{\eta(\tau)^3 \, \vartheta_{11}(\tau, 2z)}{
\vartheta_{11}(\tau, \, z+a\tau+b) \, 
\vartheta_{11}(\tau, \, z-a\tau-b)}
\, + \, 
\frac12 \, q^{a^2} \, 
\sum_{k=0}^{4a} \, (-1)^{2bk} \, 
q^{-\frac14(k-2a)^2} \, \theta_{k,1}(\tau,2z)
\label{App:eqn:2023-502a2}
\end{equation}
\end{enumerate}
\end{lemma}

\begin{proof} 1) \,\ By \eqref{App:eqn:2023-503a1} and 
\eqref{App:eqn:2023-503a2}, one has 
{\allowdisplaybreaks
\begin{eqnarray}
& & \hspace{-20mm}
\Phi^{[1,0]}_1(\tau, \, z+a\tau+b, \, z-a\tau-b, \, 0)
\,\ = \,\ 
\sum_{j \in \zzz} \, \dfrac{e^{4\pi ijz}q^{j^2}}{
1-(-1)^{2b} \, e^{2\pi iz} \, q^{j+a}}
\label{App:eqn:2023-502b}
\\[2mm]
& & \hspace{-20mm}
\Phi^{[1,0]}_2(\tau, \, z+a\tau+b, \, z-a\tau-b, \, 0)
\,\ = \,\ 
\sum_{j \in \zzz}
\dfrac{e^{-4\pi ijz}q^{j^2}}{1-e^{-2\pi i(z-a\tau-b)}q^j}
\nonumber
\\[1mm]
&=&
- \, \sum_{j \in \zzz} \, \frac{e^{4\pi ijz}q^{j^2}}{
1-(-1)^{2b} \, e^{2\pi iz} \, q^{j+a}} 
\, \cdot \, \big((-1)^{2b} \, e^{2\pi iz} \, q^{j+a}\big)^{4a+1}
\nonumber
\end{eqnarray}}
so 
\begin{subequations}
{\allowdisplaybreaks
\begin{eqnarray}
& & \hspace{-10mm}
\big[\Phi^{[1,0]}_1+\Phi^{[1,0]}_2\big]
(\tau, \, z+a\tau+b, \, z-a\tau-b, \, 0)
\, = \, 
\sum_{j \in \zzz} \, e^{4\pi ijz}q^{j^2} \cdot 
\frac{1- \big((-1)^{2b} e^{2\pi iz} q^{j+a}\big)^{4a+1}
}{1-(-1)^{2b} e^{2\pi iz} q^{j+a}} 
\nonumber
\\[1mm]
&=&
\sum_{j \in \zzz} \, e^{4\pi ijz}q^{j^2}
\sum_{k=0}^{4a}\big((-1)^{2b} \, e^{2\pi iz} \, q^{j+a}\big)^k
\, = \, 
\sum_{k=0}^{4a} \, (-1)^{2bk} \, q^{-\frac{k^2}{4}+ak} \, 
\sum_{j \in \zzz} 
e^{4\pi i (j+\frac{k}{2})z} \, q^{(j+\frac{k}{2})^2}
\nonumber
\\[1mm]
&=& 
q^{a^2} \, \sum_{k=0}^{4a} \, (-1)^{2bk} \, q^{-\frac14(k-2a)^2} \, 
\theta_{k,1}(\tau, 2z)
\label{App:eqn:2023-502c1}
\end{eqnarray}}
Also, by the $\widehat{sl}(2|1)$-denominator identity in Lemma 2.7 
in \cite{W2022a}, one has 
\begin{equation}
\big[\Phi^{[1,0]}_1-\Phi^{[1,0]}_2\big]
(\tau, \, z+a\tau+b, \, z-a\tau-b, \, 0)
\, = \, 
- i \cdot \dfrac{\eta(\tau)^3 \, \vartheta_{11}(\tau, 2z)}{
\vartheta_{11}(\tau, \, z+a\tau+b) \, 
\vartheta_{11}(\tau, \, z-a\tau-b)}
\label{App:eqn:2023-502c2}
\end{equation}
\end{subequations}
Then the claim 1) follows from \eqref{App:eqn:2023-502c1} and 
\eqref{App:eqn:2023-502c2}, and 2) follows from 1) and 
\eqref{App:eqn:2023-502b}.
\end{proof}

\vspace{1mm}



\begin{lemma} \,\ 
\label{App:lemma:2023-501b}
The following formulas hold for $n \in \zzz_{\geq 0}$.
\begin{subequations}
\begin{enumerate}
\item[{\rm (i)}] \,\ $\Phi^{[1,0]}_1(\tau, \, z+n\tau, \, z-n\tau, \, 0)$
\begin{equation}
= \,\ 
\frac12 \, q^{n^2} \, \Bigg\{- \, i \, 
\dfrac{\eta(\tau)^3 \, \vartheta_{11}(\tau, 2z)}{\vartheta_{11}(\tau,z)^2}
\, + \, 
\sum_{k=0}^{4n} \, q^{-\frac14(k-2n)^2} \, \theta_{k,1}(\tau,2z)\Bigg\}
\label{App:eqn:2023-502d1}
\end{equation}
\item[{\rm (ii)}] $\Phi^{[1,0]}_1(\tau, \, z+n\tau+\frac12, \, z-n\tau-\frac12, \, 0)$
\begin{equation}
= \, \frac12 \, q^{n^2} \, \Bigg\{i \, 
\dfrac{\eta(\tau)^3 \, \vartheta_{11}(\tau, 2z)}{
\vartheta_{10}(\tau,z)^2}
\, + \, 
\sum_{k=0}^{4n} (-1)^k \, q^{-\frac14(k-2n)^2} \, 
\theta_{k,1}(\tau,2z)\Bigg\}
\label{App:eqn:2023-502d2}
\end{equation}
\item[{\rm (iii)}] $\Phi^{[1,0]}_1(\tau, \, z+(n+\frac12)\tau, \, z-(n+\frac12)\tau, \, 0)$
\begin{equation}
= \,\ 
\frac12 \, q^{(n+\frac12)^2} \, \Bigg\{- \, i \, 
\dfrac{\eta(\tau)^3 \, \vartheta_{11}(\tau, 2z)}{
\vartheta_{01}(\tau,z)^2}
\, + \, 
\sum_{k=0}^{4n+2} \, q^{-\frac14(k-2n-1)^2} \, \theta_{k,1}(\tau,2z)\Bigg\}
\label{App:eqn:2023-502d3}
\end{equation}
\item[{\rm (iv)}] $\Phi^{[1,0]}_1(\tau, \, z+(n+\frac12)\tau+\frac12, \, 
z-(n+\frac12)\tau-\frac12, \, 0)$
\begin{equation}
= \,\ 
\frac12 \, q^{(n+\frac12)^2} \, \Bigg\{i \, 
\dfrac{\eta(\tau)^3 \, \vartheta_{11}(\tau, 2z)}{
\vartheta_{00}(\tau,z)^2}
\, + \, 
\sum_{k=0}^{4n+2} \, (-1)^k \, q^{-\frac14(k-2n-1)^2} \, 
\theta_{k,1}(\tau,2z)\Bigg\}
\label{App:eqn:2023-502d4}
\end{equation}
\end{enumerate}
\end{subequations}
\end{lemma}

\vspace{-1mm}

\begin{proof} \,\ These formulas are obtained easily from 
\eqref{App:eqn:2023-502a1}.
\end{proof}

\vspace{1mm}



\begin{prop} \,\ 
\label{cor:2023-414a}
\label{App:prop:2023-501a}
The following formulas hold for $n \in \zzz_{\geq 0}$.
\begin{enumerate}
\item[{\rm 1)}] 
\begin{subequations}
\begin{enumerate}
\item[{\rm (i)}] \quad $2 \, \sum\limits_{j \in \zzz}
\dfrac{e^{2\pi ijz} q^{\frac12j^2}}{1-e^{2\pi iz}q^{j+n}}$
\begin{equation}
= \,\ 
q^{\frac12n^2} \, \bigg\{- \, i \, 
\vartheta_{00}(\tau,0)
\cdot 
\frac{\vartheta_{01}(\tau,z) \, \vartheta_{10}(\tau,z)
}{\vartheta_{11}(\tau,z)}
\, + \, 
\sum_{k=0}^{2n} \, q^{-\frac12(k-n)^2} \, \vartheta_{00}(\tau,z) 
\bigg\}
\label{eqn:2023-425a1}
\end{equation}
\item[{\rm (ii)}] \,\ $2 \, \sum\limits_{j \in \zzz}
\dfrac{e^{2\pi i(j-\frac12)z} q^{\frac12(j-\frac12)^2}
}{1-e^{2\pi iz}q^{j+n}}$
\begin{equation}
= \,\ 
q^{\frac12(n+\frac12)^2} \, \bigg\{- \, i \,  
\vartheta_{10}(\tau,0)
\cdot 
\frac{\vartheta_{01}(\tau,z) \, \vartheta_{00}(\tau,z)
}{\vartheta_{11}(\tau,z)}
\, + \, 
\sum_{k=0}^{2n+1} \, q^{-\frac12(k-n-\frac12)^2} \, \vartheta_{10}(\tau,z)
\bigg\}
\label{eqn:2023-425a2}
\end{equation}
\item[{\rm (iii)}] \,\ $2 \, \sum\limits_{j \in \zzz}
\dfrac{e^{2\pi ijz} q^{\frac12j^2}}{1-e^{2\pi iz}q^{j+n+\frac12}}$
\begin{equation}
= \,\ 
q^{\frac12(n+\frac12)^2} \, \bigg\{
- \, i \, 
\vartheta_{10}(\tau,0)
\cdot 
\frac{\vartheta_{10}(\tau,z) \, \vartheta_{11}(\tau,z)
}{\vartheta_{01}(\tau,z)}
\, + \, 
\sum_{k=0}^{2n+1} \, q^{-\frac12(k-n-\frac12)^2} \, \vartheta_{00}(\tau,z)
\bigg\}
\label{eqn:2023-425a3}
\end{equation}
\item[{\rm (iv)}] \quad $2 \, \sum\limits_{j \in \zzz}
\dfrac{e^{2\pi i(j+\frac12)z} q^{\frac12(j+\frac12)^2}}{
1-e^{2\pi iz}q^{j+n+\frac12}}$
\begin{equation}
= \,\ 
q^{\frac12n^2} \, \bigg\{
- \, i \, 
\vartheta_{00}(\tau,0)
\cdot 
\frac{\vartheta_{00}(\tau,z) \, \vartheta_{11}(\tau,z)
}{\vartheta_{01}(\tau,z)}
\, + \, 
\sum_{k=0}^{2n} \, q^{-\frac12(k-n)^2} \, \vartheta_{10}(\tau,z)
\bigg\}
\label{eqn:2023-425a4}
\end{equation}
\end{enumerate}
\end{subequations}
\item[{\rm 2)}]
\begin{subequations}
\begin{enumerate}
\item[{\rm (i)}] \quad $2 \, \sum\limits_{j \in \zzz} (-1)^j
\dfrac{e^{2\pi ijz} q^{\frac12j^2}}{1+e^{2\pi iz}q^{j+n}}$
\begin{equation}
= \,\ 
q^{\frac12n^2} \, \bigg\{i \, 
\vartheta_{00}(\tau,0)
\cdot 
\frac{\vartheta_{00}(\tau,z) \, \vartheta_{11}(\tau,z)
}{\vartheta_{10}(\tau,z)}
+  
\sum_{k=0}^{2n} \, q^{-\frac12(k-n)^2} \, \vartheta_{01}(\tau,z) 
\bigg\}
\label{eqn:2023-425b1}
\end{equation}
\item[{\rm (ii)}] \quad $2 \, \sum\limits_{j \in \zzz} \, 
(-1)^j \, \dfrac{e^{2\pi i(j-\frac12)z} \, q^{\frac12(j-\frac12)^2} 
}{1+e^{2\pi iz}q^{j+n}}$
\begin{equation}
= \,\ 
q^{\frac12(n+\frac12)^2} \, \bigg\{
- \, 
\vartheta_{10}(\tau,0)
\cdot 
\frac{\vartheta_{01}(\tau,z) \, \vartheta_{00}(\tau,z)
}{\vartheta_{10}(\tau,z)}
+ \, i 
\sum_{k=0}^{2n+1} \, q^{-\frac12(k-n-\frac12)^2} \, 
\vartheta_{11}(\tau,z)
\bigg\}
\label{eqn:2023-425b2}
\end{equation}
\item[{\rm (iii)}] \quad $2 \, \sum\limits_{j \in \zzz} \, (-1)^j \, 
\dfrac{e^{2\pi ijz} q^{\frac12j^2}}{1+e^{2\pi iz}q^{j+n+\frac12}}$
\begin{equation}
= \,\ 
q^{\frac12(n+\frac12)^2} \, \bigg\{
i \, 
\vartheta_{10}(\tau,0)
\cdot 
\frac{\vartheta_{11}(\tau,z) \, \vartheta_{10}(\tau,z)
}{\vartheta_{00}(\tau,z)}
\,\ + \,\ 
\sum_{k=0}^{2n+1} \, q^{-\frac12(k-n-\frac12)^2} \, \vartheta_{01}(\tau,z)
\bigg\}
\label{eqn:2023-425b3}
\end{equation}
\item[{\rm (iv)}] \quad $2 \, \sum\limits_{j \in \zzz} (-1)^j \, 
\dfrac{e^{2\pi i(j+\frac12)z}q^{\frac12(j+\frac12)^2}
}{1+e^{2\pi iz}q^{j+n+\frac12}}$
\begin{equation}
= \,\ 
q^{\frac12n^2} \, \bigg\{
\vartheta_{00}(\tau,0)
\cdot 
\frac{\vartheta_{01}(\tau,z) \, \vartheta_{10}(\tau,z)
}{\vartheta_{00}(\tau,z)}
\, - \,\ i \,  
\sum_{k=0}^{2n} \, q^{-\frac12(k-n)^2} \, \vartheta_{11}(\tau,z)
\bigg\}
\label{eqn:2023-425b4}
\end{equation}
\end{enumerate}
\end{subequations}
\end{enumerate}
\end{prop}

\begin{proof} 1) \,\ We recall Lemma 2.5 in \cite{W2022a} which gives the 
following formulas for $a \in \ccc$ :
{\allowdisplaybreaks
\begin{eqnarray*}
& & \hspace{-25mm}
\Phi^{[1,0]}_1(\tau, \, z+a\tau, \, z-a\tau, \, 0)
\, + \, 
\Phi^{[1,0]}_1(\tau, \, z+a\tau+\tfrac12, \, z-a\tau-\tfrac12, \, 0)
\\[0mm]
&=&
2 \, \Phi^{[\frac12,0]}_1(2\tau, \, 2(z+a\tau), \, 2(z-a\tau), \, 0)
\\[1mm]
& & \hspace{-25mm}
\Phi^{[1,0]}_1(\tau, \, z+a\tau, \, z-a\tau, \, 0)
\, - \, 
\Phi^{[1,0]}_1(\tau, z+a\tau+\tfrac12, z-a\tau-\tfrac12,0)
\\[0mm]
&=&
2 \, \Phi^{[\frac12,\frac12]}_1(2\tau, \, 2(z+a\tau), \, 2(z-a\tau), \, 0)
\end{eqnarray*}}
Then, by using the formulas (10.5) and (10.9) in \cite{W2023a} 
and Note 1.1 in \cite{W2022c} and also using the formulas 
\begin{equation} \left\{
\begin{array}{ccl}
\vartheta_{00}(\tau,z)^2 \, + \, \vartheta_{01}(\tau,z)^2
&=& 
2 \, \eta(2\tau) \, \bigg[\dfrac{
\eta(2\tau)^2}{\eta(\tau) \, \eta(4\tau)}\bigg]^2 \, 
\vartheta_{00}(2\tau,2z)
\\[4mm]
\vartheta_{00}(\tau,z)^2 \, - \, \vartheta_{01}(\tau,z)^2
&=&
4 \, \eta(2\tau) \, \bigg[\dfrac{\eta(4\tau)}{\eta(2\tau)}\bigg]^2 \, 
\vartheta_{10}(2\tau,2z)
\\[4mm]
\vartheta_{10}(\tau,z)^2 \, + \, \vartheta_{11}(\tau,z)^2
&=& 
4 \, \eta(2\tau) \, \bigg[\dfrac{\eta(4\tau)}{\eta(2\tau)}\bigg]^2 \, 
\vartheta_{00}(2\tau,2z)
\\[4mm]
\vartheta_{10}(\tau,z)^2 \, - \, \vartheta_{11}(\tau,z)^2
&=& 
2 \, \eta(2\tau) \, \bigg[\dfrac{
\eta(2\tau)^2}{\eta(\tau) \, \eta(4\tau)}\bigg]^2 \, 
\vartheta_{10}(2\tau,2z)
\end{array} \right.
\label{App:eqn:2023-508a}
\end{equation}
and
\begin{equation}
\left\{
\begin{array}{lcc}
\vartheta_{00}(\tau,0) &=&
\dfrac{\eta(\tau)^5}{\eta(\frac{\tau}{2})^2 \, \eta(\tau)^2} 
\\[5mm]
\vartheta_{01}(\tau,0) &=&
\dfrac{\eta(\frac{\tau}{2})^2}{\eta(\tau)}
\end{array}\right. \hspace{10mm} \left\{
\begin{array}{lcc}
\vartheta_{10}(\tau,0) &=& 2 \cdot \dfrac{\eta(2\tau)^2}{\eta(\tau)} 
\\[5mm]
\vartheta_{11}(\tau,0) &=& 0
\end{array}\right.
\label{App:eqn:2023-508b}
\end{equation}
we obtain the following :
{\allowdisplaybreaks
\begin{eqnarray*}
& & \hspace{-8mm}
2 \, \Phi^{[\frac12,0]}_1(2\tau, \, 2(z+n\tau), \, 2(z-n\tau), \, 0)
\\[2mm]
& & \hspace{-7mm}
= \, 
\Phi^{[1,0]}_1(\tau, \, z+n\tau, \, z-n\tau, \, 0)
+
\Phi^{[1,0]}_1(\tau, \, z+n\tau+\tfrac12, \, z-n\tau-\tfrac12, \, 0)
\\[1mm]
& & \hspace{-7mm}
= \, 
\frac12 \, q^{n^2} \, \Bigg\{- \, i \, 
\dfrac{\eta(\tau)^3 \, \vartheta_{11}(\tau, 2z)}{\vartheta_{11}(\tau,z)^2}
+ \, i \, 
\dfrac{\eta(\tau)^3 \, \vartheta_{11}(\tau, 2z)}{\vartheta_{10}(\tau,z)^2}
\bigg\}
\\[1mm]
& &
+ \,\ \frac12 \, q^{n^2} \, 
\sum_{k=0}^{4n} \, \Big\{
q^{-\frac14(k-2n)^2} \, \theta_{k,1}(\tau,2z)
\, + \, 
(-1)^k \, q^{-\frac14(k-2n)^2} \, \theta_{k,1}(\tau,2z)
\Big\}
\\[1mm]
& & \hspace{-7mm}
= \, 
- \, i \, q^{n^2} \, 
\frac{\eta(2\tau)^5}{\eta(\tau)^2 \, \eta(4\tau)^2}
\cdot 
\frac{\vartheta_{01}(2\tau,2z) \, \vartheta_{10}(2\tau,2z)
}{\vartheta_{11}(2\tau,2z)}
\, + \, 
q^{n^2} \, \sum_{k=0}^{2n} \, q^{-(k-n)^2} \, 
\theta_{0,\frac12}(2\tau,4z)
\\[1mm]
& & \hspace{-7mm}
= \, 
- \, i \, q^{n^2} \, \vartheta_{00}(2\tau,0) 
\cdot 
\frac{\vartheta_{01}(2\tau,2z) \, \vartheta_{10}(2\tau,2z)
}{\vartheta_{11}(2\tau,2z)}
\, + \, 
q^{n^2} \, \sum_{k=0}^{2n} \, q^{-(k-n)^2} \, 
\vartheta_{00}(2\tau,2z)
\\[2mm] 
& & \hspace{-8mm}
2 \, \Phi^{[\frac12,\frac12]}_1(2\tau, \, 2(z+n\tau), \, 2(z-n\tau), \, 0)
\\[2mm]
& & \hspace{-7mm}
= \, 
\Phi^{[1,0]}_1(\tau, \, z+n\tau, \, z-n\tau, \, 0)
-
\Phi^{[1,0]}_1(\tau, \, z+n\tau+\tfrac12, \, z-n\tau-\tfrac12, \, 0)
\\[1mm]
& & \hspace{-7mm}
= \, 
\frac12 \, q^{n^2} \, \Bigg\{- \, i \, 
\dfrac{\eta(\tau)^3 \, \vartheta_{11}(\tau, 2z)}{\vartheta_{11}(\tau,z)^2}
- \, i \, 
\dfrac{\eta(\tau)^3 \, \vartheta_{11}(\tau, 2z)}{\vartheta_{10}(\tau,z)^2}
\bigg\}
\\[2mm]
& &
+ \,\ \frac12 \, q^{n^2} \, 
\sum_{k=0}^{4n} \, \Big\{
q^{-\frac14(k-2n)^2} \, \theta_{k,1}(\tau,2z)
\, - \, 
(-1)^k \, q^{-\frac14(k-2n)^2} \, \theta_{k,1}(\tau,2z)
\Big\}
\\[1mm]
& & \hspace{-7mm}
= \, - \, i \, q^{n^2} \, 
\vartheta_{10}(2\tau,0)
\cdot 
\frac{\vartheta_{01}(2\tau,2z) \, \vartheta_{00}(2\tau,2z)
}{\vartheta_{11}(2\tau,2z)}
\, + \, 
q^{n^2} \, \sum_{k=0}^{2n-1} \, q^{-(k-n+\frac12)^2} \, 
\vartheta_{10}(2\tau,2z)
\\[2mm]
& & \hspace{-8mm}
2 \, \Phi^{[\frac12,0]}_1(2\tau, \, 2(z+(n+\tfrac12)\tau), \, 
2(z-(n+\tfrac12)\tau), \, 0)
\\[2mm]
& & \hspace{-7mm}
= \, 
\Phi^{[1,0]}_1(\tau, \, z+(n+\tfrac12)\tau, \, z-(n+\tfrac12)\tau, \, 0)
+
\Phi^{[1,0]}_1(\tau, \, z+(n+\tfrac12)\tau+\tfrac12, \, 
z-(n+\tfrac12)\tau-\tfrac12, \, 0)
\\[1mm]
& & \hspace{-7mm}
= \, 
\frac12 \, q^{(n+\frac12)^2} \, \Bigg\{- \, i \, 
\dfrac{\eta(\tau)^3 \, \vartheta_{11}(\tau, 2z)}{\vartheta_{01}(\tau,z)^2}
+ \, i \, 
\dfrac{\eta(\tau)^3 \, \vartheta_{11}(\tau, 2z)}{\vartheta_{00}(\tau,z)^2}
\bigg\}
\\[2mm]
& &
+ \,\ \frac12 \, q^{(n+\frac12)^2} \, 
\sum_{k=0}^{4n+2} \, \Big\{
q^{-\frac14(k-2n-1)^2} \, \theta_{k,1}(\tau,2z)
\, + \, 
(-1)^k \, q^{-\frac14(k-2n-1)^2} \, \theta_{k,1}(\tau,2z)
\Big\}
\\[1mm]
& & \hspace{-7mm}
= \, - \, i \, q^{(n+\frac12)^2} \, 
\vartheta_{10}(2\tau,0)
\cdot 
\frac{\vartheta_{10}(2\tau,2z) \, \vartheta_{11}(2\tau,2z)
}{\vartheta_{01}(2\tau,2z)}
\, + \, 
q^{(n+\frac12)^2} \, \sum_{k=0}^{2n+1} \, q^{-(k-n-\frac12)^2} \, 
\vartheta_{00}(2\tau,2z)
\\[2mm]
& & \hspace{-8mm}
2 \, \Phi^{[\frac12,\frac12]}_1(2\tau, \, 2(z+(n+\tfrac12)\tau), \, 
2(z-(n+\tfrac12)\tau), \, 0)
\\[2mm] 
& & \hspace{-7mm}
= \, 
\Phi^{[1,0]}_1(\tau, \, z+(n+\tfrac12)\tau, \, z-(n+\tfrac12)\tau, \, 0)
-
\Phi^{[1,0]}_1(\tau, \, z+(n+\tfrac12)\tau+\tfrac12, \, 
z-(n+\tfrac12)\tau-\tfrac12, \, 0)
\\[2mm]
& & \hspace{-7mm}
= \, \frac12 \, q^{(n+\frac12)^2} \, \Bigg\{- \, i \, 
\dfrac{\eta(\tau)^3 \, \vartheta_{11}(\tau, 2z)}{\vartheta_{01}(\tau,z)^2}
- \, i \, 
\dfrac{\eta(\tau)^3 \, \vartheta_{11}(\tau, 2z)}{\vartheta_{00}(\tau,z)^2}
\bigg\}
\\[2mm]
& &
+ \,\ \frac12 \, q^{(n+\frac12)^2} \, 
\sum_{k=0}^{4n+2} \, \Big\{
q^{-\frac14(k-2n-1)^2} \, \theta_{k,1}(\tau,2z)
\, - \, 
(-1)^k \, q^{-\frac14(k-2n-1)^2} \, \theta_{k,1}(\tau,2z)
\Big\}
\\[2mm]
& & \hspace{-7mm}
= \, - \, i \, q^{(n+\frac12)^2} \, 
\vartheta_{00}(2\tau,0)
\cdot 
\frac{\vartheta_{00}(2\tau,2z) \, \vartheta_{11}(2\tau,2z)
}{\vartheta_{01}(2\tau,2z)}
\, + \, 
q^{(n+\frac12)^2} \, \sum_{k=0}^{2n} \, q^{-(k-n)^2} \, 
\vartheta_{10}(2\tau,2z)
\end{eqnarray*}}
The LHS's of the above equations are obtained from 
\eqref{App:eqn:2023-503a1} as follows:
{\allowdisplaybreaks
\begin{eqnarray*}
& & \hspace{-9mm}
2 \, \Phi^{[\frac12,0]}_1(2\tau, \, 2(z+n\tau), \, 2(z-n\tau), \, 0)
\,\ = \,\ 
2 \, \bigg[\sum_{j \in \zzz}
\frac{e^{2\pi ijz} q^{\frac12j^2}}{1-e^{2\pi iz}q^{j+n}}\bigg]_{
\substack{\tau \, \rightarrow \, 2\tau \\[1mm] z \, \rightarrow \, 2z
}}
\\[2mm] 
& & \hspace{-9mm}
2 \, \Phi^{[\frac12,\frac12]}_1(2\tau, \, 2(z+n\tau), \, 2(z-n\tau), \, 0)
\,\ = \,\ 
2 \, \bigg[ q^{\frac12n} e^{\pi iz}
\sum_{j \in \zzz}
\frac{e^{2\pi ijz} q^{\frac12j(j+1)}}{1-e^{2\pi iz}q^{j+n}}\bigg]_{
\substack{\tau \, \rightarrow \, 2\tau \\[1mm] z \, \rightarrow \, 2z
}}
\\[2mm]
& & \hspace{-9mm}
2 \, \Phi^{[\frac12,0]}_1(2\tau, \, 2(z+(n+\tfrac12)\tau), \, 
2(z-(n+\tfrac12)\tau), \, 0) 
\,\ = \,\ 
2 \, \bigg[\sum_{j \in \zzz}
\frac{e^{2\pi ijz} q^{\frac12j^2}}{1-e^{2\pi iz}q^{j+n+\frac12}}\bigg]_{
\substack{\tau \, \rightarrow \, 2\tau \\[1mm] z \, \rightarrow \, 2z
}}
\\[2mm] 
& & \hspace{-9mm}
2 \, \Phi^{[\frac12,\frac12]}_1(2\tau, \, 2(z+(n+\tfrac12)\tau), \, 
2(z-(n+\tfrac12)\tau), \, 0) 
=
2 \, \bigg[ q^{\frac12(n+\frac12)} e^{\pi iz}
\sum_{j \in \zzz}
\frac{e^{2\pi ijz} q^{\frac12j(j+1)}}{1-e^{2\pi iz}q^{j+n+\frac12}}\bigg]_{
\substack{\tau \, \rightarrow \, 2\tau \\[1mm] z \, \rightarrow \, 2z
}}
\end{eqnarray*}}
Thus the formulas in 1) are proved.
The formulas in 2) are obtained from those in 1) 
by replacing $z$ with $z+\frac12$.
\end{proof}

\section{$\Phi^{(-)[\frac12,\frac12]}_1(\tau, \, z+a\tau+b, \, z-a\tau-b, \, 0)$}
\label{sec:Phi1(-)(1/2,1/2)}


\begin{lemma} 
\label{App:lemma:2023-503a}
\begin{enumerate}
\item[{\rm 1)}] \,\ Let $a \in \frac12 \zzz_{\geq 0}$ and $b \in \frac12\zzz$ 
\,\ such that \,\ $(a+\frac12)b \, \in \, \frac12 \zzz$. Then
\begin{eqnarray*}
& & \hspace{-10mm}
\big[\Phi^{(-)[\frac12,\frac12]}_1
\, + \, (-1)^{2a} \, \Phi^{(-)[\frac12,\frac12]}_2\big]
(\tau, \, z+a\tau+b, \, z-a\tau-b, \, 0)
\\[1mm]
&=&
\left\{
\begin{array}{lcl}
0 & & {\rm if} \quad a \, \in \, \zzz \\[2mm]
- i \, e^{\pi ib} \, q^{\frac12a^2}
\sum\limits_{0 \, \leq \, k \, \leq \, 2a-1} \, (-1)^{(2b+1)k} \, 
q^{-\frac12(k-a+\frac12)^2} \, 
\vartheta_{11}(\tau,z)
& & {\rm if} \quad a \, \in \, \frac12\zzz_{\rm odd}
\end{array} \right.
\end{eqnarray*}
\item[{\rm 2)}] \,\ Let $a \in \frac12 \nnn_{\rm odd}$ and 
$b \in \frac12\zzz$. Then
{\allowdisplaybreaks
\begin{eqnarray}
& & \hspace{-10mm}
\Phi^{(-)[\frac12,\frac12]}_1
(\tau, \, z+a\tau+b, \, z-a\tau-b, \, 0)
\nonumber
\\[2mm]
&=& \hspace{5mm}
\frac{i}{2} \,\ \dfrac{\eta(\tau)^3 \, 
\vartheta_{11}(\tau, 2z) \, \vartheta_{11}(\tau, a\tau+b)
}{
\vartheta_{11}(\tau, \, z+a\tau+b) \, 
\vartheta_{11}(\tau, \, z-a\tau-b) \, \vartheta_{11}(\tau,z)}
\hspace{40mm}
\nonumber
\\[2mm]
& & - \,\ 
\frac{i}{2} \,\ e^{\pi ib} \, q^{\frac{a^2}{2}} 
\sum_{k=0}^{2a-1} \, (-1)^{(2b+1)k} \, 
q^{-\frac12(k-a+\frac12)^2} \, \vartheta_{11}(\tau,z)
\label{App:eqn:2023-503b}
\end{eqnarray}}
\end{enumerate}
\end{lemma}

\begin{proof} 1) \,\ By \eqref{App:eqn:2023-503a1} and 
\eqref{App:eqn:2023-503a2}, one has 
{\allowdisplaybreaks
\begin{eqnarray}
& & \hspace{-15mm}
\Phi^{(-)[\frac12,\frac12]}_i(\tau, \, z+a\tau+b, \, z-a\tau-b, \, 0)
\, = \, 
\sum_{j \in \zzz} (-1)^j \, 
\frac{\displaystyle e^{2\pi ijz} e^{\pi i(z+a\tau+b)}
q^{\frac12 (j^2+j)}
}{1-e^{2\pi i(z+a\tau+b)} \, q^j}
\nonumber
\\[2mm]
&=&
e^{\pi ib} \, q^{\frac{a}{2}} \, e^{\pi iz} \, 
\sum_{j \in \zzz} \, (-1)^j \, 
\frac{e^{2\pi ijz}q^{\frac12(j^2+j)}}{
1-(-1)^{2b} \, e^{2\pi iz} \, q^{j+a}}
\label{App:eqn:2023-504b}
\\[2mm]
& & \hspace{-15mm}
\Phi^{(-)[\frac12,\frac12]}_2(\tau, \, z+a\tau+b, \, z-a\tau-b, \, 0)
\,\ = \,\ 
\sum_{j \in \zzz} (-1)^j
\frac{\displaystyle e^{-2\pi ijz} \, e^{-\pi i(z-a\tau-b)}
q^{\frac12(j^2+j)}}{1-e^{-2\pi i(z-a\tau-b)}q^j}
\nonumber
\\[2mm]
&=&
e^{\pi ib} \, q^{\frac{a}{2}} \, 
\sum_{j \in \zzz} \, (-1)^j \, 
\frac{e^{-2\pi ijz} \, e^{-\pi iz} \, q^{\frac12(j^2+j)}}{
1-(-1)^{2b} \, e^{-2\pi iz} \, q^{j+a}}
\nonumber
\\[2mm]
&=&
- \, (-1)^{2b} \, e^{\pi ib} \, q^{\frac{a}{2}} \, e^{\pi iz} \, 
\sum_{j \in \zzz} \, (-1)^j \, 
\frac{e^{-2\pi ijz} \, q^{\frac12(j^2+j)-(j+a)}
}{1-(-1)^{2b} \, e^{2\pi iz} \, q^{-(j+a)}}
\nonumber
\\[2mm]
&=&
- \, (-1)^{2a} \, (-1)^{2b} \, e^{\pi ib} \, 
q^{\frac{a}{2}} \, e^{\pi iz} \, 
\sum_{j \in \zzz} \, (-1)^j \, 
\frac{e^{2\pi ijz}q^{\frac12(j^2+j)}}{
1-(-1)^{2b} \, e^{2\pi iz} \, q^{j+a}} 
\, \cdot \, (e^{2\pi iz} \, q^{j+a})^{2a}
\nonumber
\end{eqnarray}}
so
{\allowdisplaybreaks
\begin{eqnarray*}
& & \hspace{-10mm}
\Phi^{(-)[\frac12,\frac12]}_1(\tau, \, z+a\tau+b, \, z-a\tau-b, \, 0)
\, + \, (-1)^{2a} \, 
\Phi^{(-)[\frac12,\frac12]}_2(\tau, \, z+a\tau+b, \, z-a\tau-b, \, 0)
\\[3mm]
&=& \hspace{4mm}
e^{\pi ib} \, q^{\frac{a}{2}} \, e^{\pi iz} \, 
\sum_{j \in \zzz} \, (-1)^j \, 
\frac{e^{2\pi ijz}q^{\frac12(j^2+j)}}{
1-(-1)^{2b} \, e^{2\pi iz} \, q^{j+a}} 
\\[2mm]
& &- \,\ 
e^{\pi ib} \, q^{\frac{a}{2}} \, e^{\pi iz} \, 
\sum_{j \in \zzz} \, (-1)^j \, 
\frac{e^{2\pi ijz}q^{\frac12(j^2+j)}}{
1-(-1)^{2b} \, e^{2\pi iz} \, q^{j+a}} 
\, \cdot \, \big((-1)^{2b} \, e^{2\pi iz} \, q^{j+a}\big)^{2a}
\\[2mm]
&=&
e^{\pi ib} \, q^{\frac{a}{2}} \, e^{\pi iz} \, 
\sum_{j \in \zzz} \, (-1)^j \, e^{2\pi ijz}q^{\frac12(j^2+j)} \cdot 
\frac{1- \big((-1)^{2b} \, e^{2\pi iz} \, q^{j+a}\big)^{2a}
}{1-(-1)^{2b} \, e^{2\pi iz} \, q^{j+a}} 
\\[1mm]
&=&
e^{\pi ib} \, q^{\frac{a}{2}}
\sum_{0 \, \leq \, k \, \leq \, 2a-1} \, \sum_{j \in \zzz} \, 
(-1)^j \, (-1)^{2bk} \, e^{2\pi ijz} \, e^{2\pi ikz} \, e^{\pi iz}
q^{\frac12(j^2+j)+k(j+a)}
\\[1mm]
&=&
e^{\pi ib} \, q^{\frac12a^2}
\sum_{0 \, \leq \, k \, \leq \, 2a-1} (-1)^{2bk} \, q^{-\frac12(k-a+\frac12)^2} 
\sum_{j \in \zzz} \, (-1)^j \, 
e^{2\pi i (j+k+\frac12)z} \, q^{\frac12(j+k+\frac12)^2} 
\\[2mm]
&=&
e^{\pi ib} \, q^{\frac12a^2}
\sum_{0 \, \leq \, k \, \leq \, 2a-1} \, (-1)^{2bk} \, 
q^{-\frac12(k-a+\frac12)^2} 
\theta_{k+\frac12,\frac12}^{(-)}(\tau, 2z)
\\[2mm]
&=&
e^{\pi ib} \, q^{\frac12a^2}
\underbrace{
\sum_{0 \, \leq \, k \, \leq \, 2a-1} \, (-1)^{(2b+1)k} \, 
q^{-\frac12(k-a+\frac12)^2}}_{\substack{\hspace{6mm}
|| \,\ put \\[1mm] (A)
}} \,\ 
\underbrace{\theta_{\frac12,\frac12}^{(-)}(\tau, 2z)}_{
\substack{|| \\[1mm] {\displaystyle \hspace{1mm}
-i \vartheta_{11}(\tau,z)
}}}
\end{eqnarray*}}
Here $(A)$ is computed by putting $k \, = \, 2a-1-k'$ as follows :
{\allowdisplaybreaks
\begin{eqnarray*}
(A) &=& 
\sum_{0 \, \leq \, k' \, \leq \, 2a-1} 
(-1)^{(2b+1)(2a-1-k')} \, 
q^{-\frac12(2a-1-k'-a+\frac12)^2}
\\[2mm]
&=&
(-1)^{2a-1} 
\sum_{0 \, \leq \, k \, \leq \, 2a-1} (-1)^{(2b+1)k'} \, 
q^{-\frac12(k'-a+\frac12)^2}
\,\ = \,\ - \, (-1)^{2a} \, \times \, (A)
\end{eqnarray*}}
Thus we have \, $(A)=0$ \, if $a \in \zzz$, proving 1). 

\medskip

\noindent
2) \quad In the case $a \, \in \, \frac12 \nnn_{\rm odd}$, one has
\begin{subequations}
{\allowdisplaybreaks
\begin{eqnarray}
& & \hspace{-20mm}
\big[\Phi^{(-)[\frac12,\frac12]}_1
\, - \, \Phi^{(-)[\frac12,\frac12]}_2\big]
(\tau, \, z+a\tau+b, \, z-a\tau-b, \, 0)
\nonumber
\\[2mm]
&=&
- \, i \, 
e^{\pi ib} \, q^{\frac12a^2}
\sum\limits_{0 \, \leq \, k \, \leq \, 2a-1} \, (-1)^{(2b+1)k} \, 
q^{-\frac12(k-a+\frac12)^2} \, 
\vartheta_{11}(\tau,z)
\label{App:eqn:2023-504a1}
\end{eqnarray}}
by 1).  Also, by the formula (2.2a) in \cite{W2022e} (i.e, the 
$\widehat{osp}(3|2)$-denominator identity), one has
{\allowdisplaybreaks
\begin{eqnarray}
& & \hspace{-15mm}
\big[\Phi^{(-)[\frac12,\frac12]}_1+\Phi^{(-)[\frac12,\frac12]}_2\big]
(\tau, \, z+a\tau+b, \, z-a\tau-b, \, 0) 
\nonumber
\\[2mm]
&=&
i \, \eta(\tau)^3 \, 
\frac{\vartheta_{11}(\tau, 2z) \, \vartheta_{11}(\tau, \, a\tau+b)
}{
\vartheta_{11}(\tau, z+a\tau+b) \, 
\vartheta_{11}(\tau, z-a\tau-b) \, 
\vartheta_{11}(\tau, \, z)}
\label{App:eqn:2023-504a2}
\end{eqnarray}}
\end{subequations}
Then the claim 2) follows from \eqref{App:eqn:2023-504a1} and 
\eqref{App:eqn:2023-504a2}.
\end{proof}

\medskip



\begin{prop} \,\ 
\label{App:prop:2023-503a}
The following formulas hold for $n \in \zzz_{\geq 0}$ :
\begin{subequations}
\begin{enumerate}
\item[{\rm 1)}] \,\ $\sum\limits_{j \in \zzz} \, (-1)^j \, 
\dfrac{e^{2\pi i(j+\frac12)z}q^{\frac12(j+\frac12)^2}
}{1-e^{2\pi iz} \, q^{j+n+\frac12}}$
\begin{equation} \hspace{-10mm}
= \, \frac12 \, q^{\frac12 n^2} \, \Bigg\{
(-1)^n \, 
\vartheta_{01}(\tau,0) \cdot 
\frac{\vartheta_{00}(\tau, z) \, \vartheta_{10}(\tau, z)
}{\vartheta_{01}(\tau,z)} 
\, - \, i \, 
\sum_{k=0}^{2n} \, (-1)^k \, q^{-\frac12(k-n)^2} \, 
\vartheta_{11}(\tau,z)\Bigg\}
\label{App:eqn:2023-503c1}
\end{equation}
\item[{\rm 2)}] \quad 
$\sum\limits_{j \in \zzz} \, 
\dfrac{e^{2\pi i(j+\frac12)z} \, q^{\frac12(j+\frac12)^2}
}{1+e^{2\pi iz} \, q^{j+n+\frac12}}$
\begin{equation} \hspace{-10mm}
= \,\ 
\frac12 \, q^{\frac12 n^2} \, \Bigg\{
- i \, (-1)^n \, 
\vartheta_{01}(\tau,0) \cdot 
\frac{\vartheta_{01}(\tau, z) \, \vartheta_{11}(\tau, z)
}{\vartheta_{00}(\tau,z)}
\, + \, 
\sum_{k=0}^{2n} \, (-1)^k \, q^{-\frac12(k-n)^2} \, \vartheta_{10}(\tau,z)
\Bigg\}
\label{App:eqn:2023-503c2}
\end{equation}
\item[{\rm 3)}] \quad $\sum\limits_{j \in \zzz} \, (-1)^j \, 
\dfrac{e^{2\pi ijz} \, q^{\frac12 j^2}}{1-e^{2\pi iz} q^{j+n}}$
\begin{equation} \hspace{-10mm}
=
\frac12 \, q^{\frac12 n^2} \Bigg\{
-i \, (-1)^n \, 
\vartheta_{01}(\tau,0) \cdot 
\frac{\vartheta_{10}(\tau, z)\vartheta_{00}(\tau, z)
}{\vartheta_{11}(\tau, z)}
+ 
\sum_{k=0}^{2n} (-1)^k \, q^{-\frac12(k-n)^2} \, 
\vartheta_{01}(\tau,z) \Bigg\}
\label{App:eqn:2023-503c3}
\end{equation}
\item[{\rm 4)}] \quad $\sum\limits_{j \in \zzz} \, 
\dfrac{e^{2\pi ijz} \, q^{\frac12 j^2}}{1+e^{2\pi iz} q^{j+n}}$
\begin{equation} \hspace{-10mm}
= \,\ 
\frac12 \, q^{\frac12 n^2} \Bigg\{
i \, (-1)^n \, 
\vartheta_{01}(\tau,0) \cdot 
\frac{\vartheta_{11}(\tau,z) \, \vartheta_{01}(\tau,z)
}{\vartheta_{10}(\tau,z)}
\, + \, 
\sum_{k=0}^{2n} (-1)^k \, q^{-\frac12(k-n)^2} \, \vartheta_{00}(\tau,z)
\Bigg\}
\label{App:eqn:2023-503c4}
\end{equation}
\end{enumerate}
\end{subequations}
\end{prop}

\begin{proof} \,\  The claim 1) follows from \eqref{App:eqn:2023-503b}
and \eqref{App:eqn:2023-504b} by letting $a=n+\frac12$ and $b=0$. 
The claim 2) (resp. 3)) is obtained from 1) by replacing $z$ with 
$z+\frac12$ (resp. $z-\frac{\tau}{2}$), and the claim 4)
is obtained from 3) by replacing $z$ with $z+\frac12$. 
\end{proof}

\section{Functions $f_i(\tau,z)$, \, $g_i(\tau,z)$ and $h_i(\tau,z)$}
\label{sec:fi:gi:hi}


Letting $n=0$ in the formulas in Propositions \ref{App:prop:2023-501a}
and \ref{App:prop:2023-503a}, we consider the following functions:
{\allowdisplaybreaks
\begin{eqnarray*}
& & \hspace{-3mm}
f_1(\tau,z) := 
\sum_{j \in \zzz}\frac{e^{2\pi ijz} q^{\frac12j^2}}{1-e^{2\pi iz}q^j}
\, = \, 
\frac12 \, \bigg\{
- \, i \, \vartheta_{00}(\tau,0) \cdot 
\dfrac{\vartheta_{01}(\tau,z) \, \vartheta_{10}(\tau,z)
}{\vartheta_{11}(\tau,z)}
+  \vartheta_{00}(\tau,z)\bigg\}
\\[2mm]
& & \hspace{-3mm}
f_2(\tau,z) := \sum_{j \in \zzz} (-1)^j \, 
\dfrac{e^{2\pi i(j+\frac12)z}q^{\frac12(j+\frac12)^2}
}{1+e^{2\pi iz}q^{j+\frac12}}
\, = \, \frac12 \, \bigg\{
\vartheta_{00}(\tau,0)
\cdot 
\dfrac{\vartheta_{01}(\tau,z) \, \vartheta_{10}(\tau,z)
}{\vartheta_{00}(\tau,z)}
\, - \, i \,  \vartheta_{11}(\tau,z)\bigg\}
\\[2mm]
& & \hspace{-3mm}
f_3(\tau,z) := \sum_{j \in \zzz} (-1)^j
\dfrac{e^{2\pi ijz} q^{\frac12j^2}}{1+e^{2\pi iz}q^{j}}
\, = \,\ \frac12 \, \bigg\{
i \, \vartheta_{00}(\tau,0)
\cdot 
\dfrac{\vartheta_{00}(\tau,z) \, \vartheta_{11}(\tau,z)
}{\vartheta_{10}(\tau,z)}
\, + \, 
\vartheta_{01}(\tau,z)\bigg\}
\\[2mm]
& & \hspace{-3mm}
f_4(\tau,z) := \sum_{j \in \zzz}
\dfrac{e^{2\pi i(j+\frac12)z} q^{\frac12(j+\frac12)^2}}{
1-e^{2\pi iz}q^{j+\frac12}}
= \, \frac12 \, \bigg\{
- \, i \, \vartheta_{00}(\tau,0)
\cdot 
\dfrac{\vartheta_{00}(\tau,z) \, \vartheta_{11}(\tau,z)
}{\vartheta_{01}(\tau,z)}
\, + \, 
\vartheta_{10}(\tau,z)\bigg\}
\\[2mm]
& & \hspace{-3mm}
g_1(\tau,z) := \sum_{j \in \zzz}
\dfrac{e^{2\pi i(j-\frac12)z} q^{\frac12j(j-1)}}{1-e^{2\pi iz}q^{j}}
\, = \, 
\frac12 \, \bigg\{
- \, i \, \vartheta_{10}(\tau,0)
\cdot 
\dfrac{\vartheta_{01}(\tau,z) \, \vartheta_{00}(\tau,z)
}{\vartheta_{11}(\tau,z)}
\, + \, 
2 \, q^{-\frac18} \, \vartheta_{10}(\tau,z)\bigg\}
\\[2mm]
& & \hspace{-3mm}
g_2(\tau,z) := \sum_{j \in \zzz} 
(-1)^j \, \dfrac{e^{2\pi i(j-\frac12)z} \, q^{\frac12j(j-1)} 
}{1+e^{2\pi iz}q^{j}}
\\[2mm]
& &\hspace{11mm}
= \frac12 \, \bigg\{
- \vartheta_{10}(\tau,0)
\cdot 
\dfrac{\vartheta_{01}(\tau,z) \, \vartheta_{00}(\tau,z)
}{\vartheta_{10}(\tau,z)}
+ 
2 \, i \, q^{-\frac18} \, \vartheta_{11}(\tau,z)\bigg\}
\\[2mm]
& & \hspace{-3mm}
g_3(\tau,z) := \sum_{j \in \zzz}
\dfrac{e^{2\pi ijz} q^{\frac12(j^2-\frac14)}}{1-e^{2\pi iz}q^{j+\frac12}}
\, = \, 
\frac12 \, \bigg\{
- \, i \, \vartheta_{10}(\tau,0)
\cdot 
\dfrac{\vartheta_{10}(\tau,z) \, \vartheta_{11}(\tau,z)
}{\vartheta_{01}(\tau,z)}
\, + \, 
2 \, q^{-\frac18} \vartheta_{00}(\tau,z)\bigg\}
\\[2mm]
& & \hspace{-3mm}
g_4(\tau,z) := \sum_{j \in \zzz} \, (-1)^j \, 
\dfrac{e^{2\pi ijz} q^{\frac12(j^2-\frac14)}}{1+e^{2\pi iz}q^{j+\frac12}}
\, = \,\ \frac12 \, \bigg\{
i \, \vartheta_{10}(\tau,0)
\cdot 
\frac{\vartheta_{11}(\tau,z) \, \vartheta_{10}(\tau,z)
}{\vartheta_{00}(\tau,z)}
\, + \, 
2 \, q^{-\frac18} \, \vartheta_{01}(\tau,z)\bigg\}
\\[2mm]
& & \hspace{-3mm}
h_1(\tau,z) := \sum\limits_{j \in \zzz} \, (-1)^j \, 
\dfrac{e^{2\pi ijz} \, q^{\frac12 j^2}}{1-e^{2\pi iz} q^{j}}
\,\ = \,\ \frac12 \, \bigg\{
-i \, \vartheta_{01}(\tau,0) \cdot 
\dfrac{\vartheta_{10}(\tau, z)\vartheta_{00}(\tau, z)
}{\vartheta_{11}(\tau, z)}
\, + \, 
\vartheta_{01}(\tau,z)\bigg\}
\\[2mm]
& & \hspace{-3mm}
h_2(\tau,z) := \sum_{j \in \zzz} \, (-1)^j \, 
\dfrac{e^{2\pi i(j+\frac12)z}q^{\frac12(j+\frac12)^2}
}{1-e^{2\pi iz} \, q^{j+\frac12}}
\,\ = \,\ 
\frac12 \, \bigg\{
\vartheta_{01}(\tau,0) \cdot 
\dfrac{\vartheta_{00}(\tau, z) \, \vartheta_{10}(\tau, z)
}{\vartheta_{01}(\tau,z)} \, 
- \, 
i \, \vartheta_{11}(\tau,z)\bigg\}
\\[2mm]
& & \hspace{-3mm}
h_3(\tau,z) := \sum_{j \in \zzz} \, 
\dfrac{e^{2\pi ijz} \, q^{\frac12 j^2}}{1+e^{2\pi iz} q^{j}}
\,\ = \,\ \frac12 \, \bigg\{
i \, \vartheta_{01}(\tau,0) \cdot 
\dfrac{\vartheta_{11}(\tau,z) \, \vartheta_{01}(\tau,z)
}{\vartheta_{10}(\tau,z)}
\, + \, 
\vartheta_{00}(\tau,z)\bigg\}
\\[2mm]
& & \hspace{-3mm}
h_4(\tau,z) := \sum_{j \in \zzz} \, 
\dfrac{e^{2\pi i(j+\frac12)z} \, q^{\frac12(j+\frac12)^2}
}{1+e^{2\pi iz} \, q^{j+\frac12}}
\,\ = \,\ 
\frac12 \, \bigg\{
- i \, \vartheta_{01}(\tau,0) \cdot 
\dfrac{\vartheta_{01}(\tau, z) \, \vartheta_{11}(\tau, z)
}{\vartheta_{00}(\tau,z)}
\, + \, 
\vartheta_{10}(\tau,z)\bigg\}
\end{eqnarray*}}

\medskip

We note that the formulas in Propositions \ref{App:prop:2023-501a}
and \ref{App:prop:2023-503a} can be recovered from 
these formulas for $f_i(\tau,z)$, $g_i(\tau,z)$ and $h_i(\tau,z)$ 
by using Lemma 2.4 in \cite{W2022c}.
The modular transformation of these functions are computed
by using the modular tranformation properties of 
$\vartheta_{ab}$ in the Mumford's book \cite{Mum} and 
obtained as follows:

\vspace{1mm}

\begin{prop} 
\begin{enumerate}
\item[{\rm 1)}] \quad $S$\text{-}{\rm transformation \, :}
{\allowdisplaybreaks
\begin{eqnarray*}
& & \left\{
\begin{array}{rcl}
f_1(-\frac{1}{\tau}, \frac{z}{\tau})
&=& 
e^{\frac{\pi iz^2}{\tau}} \, \big\{
\tau \, f_1(\tau,z)
\, - \, \frac12 \, \tau \, \vartheta_{00}(\tau,z)
\, + \, \frac12 \, (-i\tau)^{\frac12} \, \vartheta_{00}(\tau,z)\big\}
\\[2mm]
f_2(-\frac{1}{\tau}, \frac{z}{\tau})
&=& 
e^{\frac{\pi iz^2}{\tau}} \, \big\{
- \, i \, \tau \, f_2(\tau,z)
\, + \, \frac12 \, \tau \, \vartheta_{11}(\tau,z)
\, - \, \frac12 \, (-i\tau)^{\frac12} \, \vartheta_{11}(\tau,z)\big\}
\\[2mm]
f_3(-\frac{1}{\tau}, \frac{z}{\tau})
&=& 
e^{\frac{\pi iz^2}{\tau}} \, \big\{
\tau \, f_4(\tau,z)
\, - \, \frac12 \, \tau \, \vartheta_{10}(\tau,z)
\, + \, \frac12 \, (-i\tau)^{\frac12} \, \vartheta_{10}(\tau,z)\big\}
\\[2mm]
f_4(-\frac{1}{\tau}, \frac{z}{\tau})
&=& 
e^{\frac{\pi iz^2}{\tau}} \, \big\{
\tau \, f_3(\tau,z)
\, - \, \frac12 \, \tau \, \vartheta_{01}(\tau,z)
\, + \, \frac12 (-i\tau)^{\frac12} \, \vartheta_{01}(\tau,z)\big\}
\end{array} \right.
\\[2mm]
& & \left\{
\begin{array}{rcl}
g_1(-\frac{1}{\tau}, \frac{z}{\tau})
&=& e^{\frac{\pi iz^2}{\tau}} \, \big\{
\tau \, h_1(\tau,z) \, - \, \frac12 \, \tau \, \vartheta_{01}(\tau,z)\big\}
\,\ + \,\ 
(-i\tau)^{\frac12} \, 
e^{\frac{\pi i}{\tau}(z^2+\frac14)} \, \vartheta_{01}(\tau,z)
\\[2mm]
g_2(-\frac{1}{\tau}, \frac{z}{\tau})
&=& 
e^{\frac{\pi iz^2}{\tau}} \, \big\{
i \tau \, h_2(\tau,z) \, - \, \frac12 \, \tau \, \vartheta_{11}(\tau,z)\big\}
\,\ + \,\ 
(-i\tau)^{\frac12} \, e^{\frac{\pi i}{\tau}(z^2+\frac14)} \, 
\vartheta_{11}(\tau,z)
\\[2mm]
g_3(-\frac{1}{\tau}, \frac{z}{\tau})
&=& 
e^{\frac{\pi iz^2}{\tau}} \, \big\{
\tau \, h_3(\tau,z) \, - \, \frac12 \, \tau \, \vartheta_{00}(\tau,z)\big\}
\,\ + \,\ 
(-i\tau)^{\frac12} \, e^{\frac{\pi i}{\tau}(z^2+\frac14)} \, 
\vartheta_{00}(\tau,z)
\\[2mm]
g_4(-\frac{1}{\tau}, \frac{z}{\tau})
&=& 
e^{\frac{\pi iz^2}{\tau}} \, \big\{
\tau \, h_4(\tau,z) \, - \, \frac12 \, \tau \, \vartheta_{10}(\tau,z)\big\}
\,\ + \,\ 
(-i\tau)^{\frac12} \, e^{\frac{\pi i}{\tau}(z^2+\frac14)} \, 
\vartheta_{10}(\tau,z)
\end{array} \right.
\\[2mm]
& & \left\{
\begin{array}{rcl}
h_1(-\frac{1}{\tau}, \frac{z}{\tau})
&=& 
e^{\frac{\pi iz^2}{\tau}} \, \big\{
\tau \, g_1(\tau,z) 
\,\ - \,\ \tau \, q^{-\frac18} \, \vartheta_{10}(\tau,z)
\,\ + \,\ 
\frac12 \, (-i\tau)^{\frac12} \, \vartheta_{10}(\tau,z)
\big\}
\\[2mm]
h_2(-\frac{1}{\tau}, \frac{z}{\tau})
&=& 
e^{\frac{\pi iz^2}{\tau}} \, \big\{
i \, \tau \, g_2(\tau,z) 
\,\ + \,\ \tau \, q^{-\frac18} \, \vartheta_{11}(\tau,z)
\,\ - \,\ 
\frac12 \, (-i\tau)^{\frac12} \, \vartheta_{11}(\tau,z)
\big\}
\\[2mm]
h_3(-\frac{1}{\tau}, \frac{z}{\tau})
&=& 
e^{\frac{\pi iz^2}{\tau}} \, \big\{
\tau \, g_3(\tau,z) 
\,\ - \,\ \tau \, q^{-\frac18} \, \vartheta_{00}(\tau,z)
\,\ + \,\ 
\frac12 \, (-i\tau)^{\frac12} \, \vartheta_{00}(\tau,z)
\big\}
\\[2mm]
h_4(-\frac{1}{\tau}, \frac{z}{\tau})
&=& 
e^{\frac{\pi iz^2}{\tau}} \, \big\{
\tau \, g_4(\tau,z) 
\,\ - \,\ \tau \, q^{-\frac18} \, \vartheta_{01}(\tau,z)
\,\ + \,\ 
\tfrac12 \, (-i\tau)^{\frac12} \, \vartheta_{01}(\tau,z)
\big\}
\end{array} \right.
\end{eqnarray*}}
\item[{\rm 2)}] \quad $T$\text{-}{\rm transformation \, :} 
{\allowdisplaybreaks
\begin{eqnarray*}
& & \hspace{-10mm}
\left\{
\begin{array}{ccr}
f_1(\tau+1,z) &=& h_1(\tau,z) \\[1mm]
f_2(\tau+1,z) &=& e^{\frac{\pi i}{4}} \, h_2(\tau,z) \\[1mm]
f_3(\tau+1,z) &=& h_3(\tau,z) \\[1mm]
f_4(\tau+1,z) &=& e^{\frac{\pi i}{4}} \, h_4(\tau,z)
\end{array}\right. \hspace{10mm} \left\{
\begin{array}{ccr}
g_1(\tau+1,z) &=& g_1(\tau,z) \\[1mm]
g_2(\tau+1,z) &=& g_2(\tau,z) \\[1mm]
g_3(\tau+1,z) &=& e^{-\frac{\pi i}{4}} \, g_4(\tau,z) \\[1mm]
g_4(\tau+1,z) &=& e^{-\frac{\pi i}{4}} \, g_3(\tau,z)
\end{array}\right.
\\[2mm]
& & \hspace{-10mm}
\left\{
\begin{array}{ccr}
h_1(\tau+1,z) &=& f_1(\tau,z) \\[1mm]
h_2(\tau+1,z) &=& e^{\frac{\pi i}{4}} \, f_2(\tau,z) \\[1mm]
h_3(\tau+1,z) &=& f_3(\tau,z) \\[1mm]
h_4(\tau+1,z) &=& e^{\frac{\pi i}{4}} \, f_4(\tau,z)
\end{array}\right.
\end{eqnarray*}}
\end{enumerate}
\end{prop}

\medskip

The asymptotics of these functions as $\tau$ tends to $0$, namely 
the asymptotic behavior when $\tau=iT$ \, $(T>0)$ and $T \rightarrow 0$, 
are given by the following:

\vspace{1mm}

\begin{prop} \,\ 
\label{App:prop:2023-507a}
For $a \in \ccc$, the asymptotics of $f_i(\tau,a\tau)$, $g_i(\tau,a\tau)$
and $h_i(\tau,a\tau)$ as $\tau \downarrow 0$ are as follows: 
{\allowdisplaybreaks
\begin{eqnarray*}
& & \left\{
\begin{array}{rcl}
f_1(\tau, a\tau) &\overset{\substack{\tau \downarrow 0 \\[0.5mm]}}{\sim} &
\frac12 (-i\tau)^{-1} \cot(a\pi)
\\[2mm]
f_2(\tau, a\tau) &\overset{\substack{\tau \downarrow 0 \\[0.5mm]}}{\sim} &
(-i\tau)^{-1} e^{-\frac{\pi i}{4\tau}} \cos(a\pi)
\\[2mm]
f_3(\tau, a\tau) &\overset{\substack{\tau \downarrow 0 \\[0.5mm]}}{\sim} &
(-i\tau)^{-1} e^{-\frac{\pi i}{4\tau}} \sin(a\pi)
\\[2mm]
f_4(\tau, a\tau) &\overset{\substack{\tau \downarrow 0 \\[0.5mm]}}{\sim} &
-\frac12 (-i\tau)^{-1} \tan(a\pi)
\end{array} \right.  \hspace{5mm} \left\{
\begin{array}{rcl}
g_1(\tau, a\tau) &\overset{\substack{\tau \downarrow 0 \\[0.5mm]}}{\sim} &
\frac12 (-i\tau)^{-1} \cot(a\pi)
\\[2mm]
g_2(\tau, a\tau) &\overset{\substack{\tau \downarrow 0 \\[0.5mm]}}{\sim} &
-(-i\tau)^{-1} e^{-\frac{\pi i}{4\tau}} \cos(a\pi)
\\[2mm]
g_3(\tau, a\tau) &\overset{\substack{\tau \downarrow 0 \\[0.5mm]}}{\sim} &
-\frac12 (-i\tau)^{-1} \tan(a\pi)
\\[2mm]
g_4(\tau, a\tau) &\overset{\substack{\tau \downarrow 0 \\[0.5mm]}}{\sim} &
(-i\tau)^{-1} e^{-\frac{\pi i}{4\tau}} \sin(a\pi)
\end{array} \right.
\\[2mm]
& & \left\{
\begin{array}{rcl}
h_1(\tau, a\tau) &\overset{\substack{\tau \downarrow 0 \\[0.5mm]}}{\sim} &
(-i\tau)^{-1} \cdot \dfrac{1}{2 \sin(a\pi)}
\\[3.5mm]
h_2(\tau, a\tau) &\overset{\substack{\tau \downarrow 0 \\[0.5mm]}}{\sim} &
(-i\tau)^{-1} \cdot \dfrac{1}{2 \cos(a\pi)}
\\[2mm]
h_3(\tau, a\tau) &\overset{\substack{\tau \downarrow 0 \\[0.5mm]}}{\sim} &
\frac12(-i\tau)^{-\frac12} 
\\[2mm]
h_4(\tau, a\tau) &\overset{\substack{\tau \downarrow 0 \\[0.5mm]}}{\sim} &
\frac12 (-i\tau)^{-\frac12} 
\end{array} \right.
\end{eqnarray*}}
\end{prop}

\vspace{-1mm}

\begin{proof} \,\ These are obtained by easy calculation 
using the asymptotic behavior of the Mumford's theta functions 
given by the following :
\begin{equation} \left\{
\begin{array}{rcl}
\vartheta_{00}(\tau, \, a\tau) 
&\overset{\substack{\tau \downarrow 0 \\[0.5mm]}}{\sim}& 
(-i\tau)^{-\frac12}
\\[2mm]
\vartheta_{01}(\tau, \, a\tau) 
&\overset{\substack{\tau \downarrow 0 \\[0.5mm]}}{\sim}&
2 \cos (a\pi) \cdot (-i\tau)^{-\frac12} \, e^{-\frac{\pi i}{4\tau}}
\\[2mm]
\vartheta_{10}(\tau, \, a\tau) 
&\overset{\substack{\tau \downarrow 0 \\[0.5mm]}}{\sim}&
(-i\tau)^{-\frac12}
\\[2mm]
\vartheta_{11}(\tau, \, a\tau) 
&\overset{\substack{\tau \downarrow 0 \\[0.5mm]}}{\sim}&
-2 i \sin (a\pi) \cdot (-i\tau)^{-\frac12} \, 
e^{-\frac{\pi i}{4\tau}}
\end{array}\right.
\label{App:eqn:2023-507a}
\end{equation}
\end{proof}

\end{document}